\documentclass[12pt]{article}
\usepackage{hyperref}
\usepackage{cleveref}
\usepackage{amssymb, amsmath, latexsym, a4}
\usepackage{graphicx}
\usepackage[all]{xy}
\usepackage{multicol}
\usepackage[usenames]{color}
 
\usepackage{amssymb}
\usepackage{mathrsfs}
\newtheorem{theorem}{Theorem}[section]
\newtheorem{lemma}[theorem]{Lemma}
\newtheorem{remark}[theorem]{Remark}
\newtheorem{corollary}[theorem]{Corollary}
\newtheorem{proposition}[theorem]{Proposition}

\newtheorem{example}[theorem]{Example}
\newenvironment{proof}{\trivlist\item[]\rm{\textbf{Proof.}\ }}{\endtrivlist}
 \makeatletter
      \def\@setcopyright{}
      \def\serieslogo@{}
      \makeatother

 \newcommand{\Ima}{\mathrm{Im}}

\newcommand{\Lie}{\ensuremath{\mathsf{Lie}}}
\newcommand{\Leib}{\ensuremath{\mathsf{Leib}}}

\newcommand{\q}{\mathfrak{q}}
\newcommand{\m}{\mathfrak{m}}
\newcommand{\g}{\mathfrak{g}}
\newcommand{\s}{\mathfrak{s}}
\newcommand{\f}{\mathfrak{f}}
\newcommand{\h}{\mathfrak{h}}
\newcommand{\p}{\mathfrak{p}}

\newcommand{\I}{\mathfrak{I}}
\newcommand{\R}{\mathfrak{r}}
\newcommand{\U}{\mathfrak{u}}

\author{Hesam Safa and Guy R. Biyogmam
    \\
}
\title{On the Schur \Lie-multiplier and \Lie-covers  of Leibniz $n$-algebras}
\begin{document}
\maketitle


\noindent\textbf{Abstract.}
In this article, we study the notion of central extension of Leibniz $n$-algebras relative to $n$-Lie algebras to study  properties  of Schur \Lie-multiplier and \Lie-covers  on Leibniz $n$-algebras.   We provide a characterization of \Lie-perfect Leibniz $n$-algebras by means of universal \Lie-central extensions. It is also provided some inequalities on the dimension of the Schur \Lie-multiplier of Leibniz $n$-algebras.   Analogue to
Wiegold \cite{w} and
Green \cite{grn} results on  groups or   Moneyhun \cite{m}  result on Lie algebras, we provide  upper bounds for the dimension of the \Lie-commutator of a Leibniz $n$-algebra with  finite dimensional \Lie-central factor, and also for  the dimension of the  Schur \Lie-multiplier  of a finite dimensional Leibniz $n$-algebra.\\
\\
\textbf{2010 MSC:} 17A32,  18B99.  \\
\textbf{Key words:} Leibniz $n$-algebra, \Lie-cover,  Schur \Lie-multiplier.


\section{Introduction}

The concept of Schur multiplier was introduced by Schur  \cite{s} in 1904 in his study of projective representations of a group. Since then, it  has appeared in many studies related to mathematical concepts such as efficient presentations, homology and projective representations of various algebraic structures.   The Schur multiplier is very useful as a tool to classify $p$-groups and nilpotent Lie algebras, thanks to the results obtained by Green \cite{grn} and Moneyhun \cite{m}, respectively.  In fact, Green  showed that   for every finite $p$-group $G$ of order
$p^k$, there is a non-negative integer $t(G)$ such that $|M(G)|=p^{\frac{1}{2}k(k-1)-t(G)},$ and  Moneyhun showed   that if $L$ is a Lie algebra of dimension $k$, then $\dim\mathcal{M}(L)\leq \frac{1}{2}k(k-1)$.
In this paper, we aim to provide a similar upper bound for Leibniz $n$-algebras, which is a generalization of both Leibniz algebras \cite{l3} and $n$-Lie algebras  (also known as Filippov algebras) \cite{Fil}. This upper bound can be useful in characterizing  \Lie-nilpotent Leibniz algebras, and more generally,  \Lie-nilpotent Leibniz $n$-algebras.

Some research interest has focussed on investigating properties on Leibniz $n$-algebras analogue to results obtained on Leibniz algebras and $n$-Lie algebras. It turns out that several of these properties fail to extend to Leibniz $n$-algebras.  The present paper provides a few counter-examples.

The concept of isoclinism was initiated in 1939 by P. Hall in his  work on the classification of $p$-groups using an equivalence relation weaker than the notion of isomorphism \cite{p-h}. Studies of this concept on various algebraic structures can be found in  \cite{Bio,  e-m-s, He,  PMKh,sal3}.

Following a philosophy that comes from  the categorical theory of central extensions relative to a chosen Birkhoff subcategory of a semi-abelian category \cite{c-v},  and applying on  the semi-abelian category of Leibniz
algebras with respect to its Birkhoff subcategory of Lie algebras, the concepts of Schur multiplier, isoclinism, and cover have been recently considered on Leibniz algebra in the relative context, i.e. with respect to  the Liezation functor  $(-)_{\Lie}: {\sf Leib} \to {\sf Lie}$ which assigns the Lie algebra $\q_{\Lie} = \dfrac{\q}{\langle [x,x]: x \in \q\rangle}$ to a given Leibniz algebra $\q.$ This yielded the notions of Schur \Lie-multiplier, \Lie-isoclinism, \Lie-cover  \cite{b-c1, b-c2, b-c3, c-i} as well as other notions such as  \Lie-stem Leibniz algebras, \Lie-perfect Leibniz algebras  and \Lie-abelian Leibniz algebras.  Note that these new notions provide a solid framework in which one can identify  properties on $n$-Lie algebras and  Leibniz algebras that cannot be extended to Leibniz $n$-algebras. For instance,
the concepts of isomorphism  and isoclinism coincide for finite dimensional stem $n$-Lie algebras (see \cite{e-m-s,m,sal2}).
However,  this does not hold for finite dimensional $\Lie$-stem Leibniz $n$-algebras ($n\geq 2$), as  discussed in \cite{r-c}. Also, the  Schur \Lie-multiplier, and more generally, the $c$-nilpotent Schur \Lie-multiplier  of a Leibniz algebra are both Lie algebras. But, this is not true for Leibniz $n$-algebras with $n\geq 3,$ as discussed in Remark \ref{rem17}. Moreover, it is shown in \cite{d-s} that a $k$-dimensional $n$-Lie algebra $L$ is abelian if and only if the dimension of its Schur multiplier $\dim\mathcal{M}(L)={k\choose n}$. This does not hold for Leibniz $n$-algebras
with $n\geq3$, as discussed in Example \ref{ex23}.


In this article, we continue the study of \Lie-central extensions on  Leibniz $n$-algebras, initiated in \cite{s-b,s-b2}.  Concretely, we study  the notions of Schur \Lie-multiplier and \Lie-cover  on Leibniz $n$-algebras, and their interplay with \Lie-isoclinism. In doing so, we organize the paper as follows: in section \ref{Preliminaries}, we recall the \Lie-notions from \cite{s-b},  define the concept of Schur \Lie-multiplier on Leibniz $n$-algebras and present some preliminaries including useful results from \cite{s-b2}.  In section \ref{Lie-cover}, we study several properties of  \Lie-covers of Leibniz $n$-algebras. In particular,
we  show that every \Lie-perfect Leibniz $n$-algebra admits at least one \Lie-cover. Also, using a homological method analogue to Loday's discussion in \cite{l2} on a necessary and sufficient condition for a pair of groups to be perfect, we  provide a characterization of \Lie-perfect Leibniz $n$-algebras by means of universal \Lie-central extensions. The last section is devoted to some inequalities on the dimension of the Schur \Lie-multiplier of Leibniz $n$-algebras. In particular, we establish  analogues of Wiegold \cite{w} (resp.  Moneyhun \cite{m})  result on an upper bound for the  order (resp. dimension) of the commutator subgroup (resp. derived algebra) of a group with finite  central factor (resp. Lie algebra with finite dimensional central factor), and also Moneyhun \cite{m} result on  an upper bound for the dimension of the Schur multiplier of a
 finite dimensional Lie algebra. More precisely, we show that for a Leibniz $n$-algebra $\q$ with $k$-dimensional \Lie-central factor  $\q/Z_\Lie(\q),$  \Lie-commutator $\q_\Lie^n$,
\[\dim\q_\Lie^n\leq \sum_{i=1}^{n} {{n-1}\choose{i-1}}{k\choose i},\]
and also for a $k$-dimensional  Leibniz $n$-algebra $\q$,   Schur \Lie-multiplier  $\mathcal{M}_\Lie(\q),$
\[\dim\mathcal{M}_\Lie(\q)\leq\sum_{i=1}^{n} {{n-1}\choose{i-1}}{k\choose i}.\]


\section{Preliminaries}\label{Preliminaries}
Throughout this paper, $n\geq 2$ is a fixed  integer and all Leibniz $n$-algebras are
considered over a fixed field $\mathbb{K}$ of characteristic zero.

 Recall that a Leibniz algebra \cite{l1, l3} is a vector space $\q$  equipped with a bilinear map $[-,-] : \q \times \q \to \q$, usually called the Leibniz bracket of ${\q}$,  satisfying the Leibniz identity:
\[
 [x,[y,z]]= [[x,y],z]-[[x,z],y], \quad x, y, z \in \q.
\]
A Leibniz $n$-algebra \cite{c-l-p} is a vector space $\q$ together with the following
$n$-linear map $[-,\ldots,-]:\q\times\cdots\times\q\longrightarrow  \q$
 such that
\[[[x_1,\ldots,x_n],y_2,\ldots,y_n]=\sum_{i=1}^{n} [x_1,\ldots,x_{i-1},[x_i,y_2,\ldots,y_n],x_{i+1},\ldots,x_n],\]
for all $x_i,y_j\in \q$, $1\leq i\leq n$ and $2\leq j\leq n$.
A subspace $\h$ of a Leibniz $n$-algebra $\q$ is called a subalgebra, if $[x_1,\ldots,x_n]\in\h$, for any $x_i\in\h$.
Also, a subalgebra $\I$ of $\q$ is said to be an ideal, if the brackets $[\I,\q,\ldots,\q]$, $[\q,\I,\q,\ldots,\q]$,
$\ldots$ ,  $[\q,\ldots,\q,\I]$ are all contained in $\I$.

Let $\q$ be a Leibniz $n$-algebra. Following \cite{s-b}, we define the bracket
\[[-,\ldots,-]_{\Lie}:\q\times\cdots\times\q\longrightarrow \q\]
\[\ \ [x_1, \ldots,x_n]_{\Lie}=\sum_{1\leq i_j\leq n}[x_{i_1},\ldots,x_{i_n}],\]
which is including $n!$ brackets. More precisely,
\begin{align}\label{eq1}
 [x_1, \ldots,x_n]_{\Lie}&=[(x_1 +\cdots+x_n),\ldots,(x_1 +\cdots+x_n)]\nonumber\\
 &-\sum_{1\leq i_j\leq n}[x_{i_1},\ldots,x_{i_j},\ldots,x_{i_j},\ldots,x_{i_n}].
\end{align}
Also, the $\Lie$-center and the $\Lie$-commutator of a Leibniz $n$-algebra $\q$ are defined as
$Z_{\Lie}(\q)=\{x\in\q| \ [x,_{n-1}\q]_\Lie=0\}$ and
$\q_{\Lie}^n=\langle [x_1,\ldots,x_n]_{\Lie} |\ x_i\in \q\rangle,$
where $[x,_{n-1}\q]_\Lie=[x,\q,\ldots,\q]_{\Lie}$. Clearly, these  are  ideals of $\q$ (see \cite[Remark 2.1]{s-b}).

Let $0\rightarrow\R\rightarrow\f\rightarrow\q\rightarrow 0$  be a free presentation of a Leibniz $n$-algebra $\q$. The  Schur \Lie-multiplier
of $\q$ is defined as
\[\mathcal{M}_\Lie(\q)=\dfrac{\R\cap\f_\Lie^n}{[\R,_{n-1}\f]_\Lie}.\]
Clearly, the Schur \Lie-multiplier is  independent of
the chosen free presentation of  $\q$. Also if $n=2$, then this definition coincides with the notion of the Schur
\Lie-multiplier of a Leibniz algebra given in \cite{c-i}.
Note that two other notions of the Schur multiplier, namely the $c$-nilpotent Schur \Lie-multiplier of a Leibniz algebra and
 the Schur multiplier  of a pair of Leibniz algebras
are already discussed in \cite{b-c3, b-s}, respectively.

\begin{remark}\label{rem17} \normalfont
In the case where the $n$-linear map $[-, \ldots , -]$ is anti-symmetric in each
pair of variables, i.e. $[x_1,  \ldots , x_i,\ldots, x_j, \ldots, x_n] = - [x_1, \ldots , x_j,\ldots, x_i, \ldots, x_n],$
or equivalently $[x_1, \ldots , x_i, \ldots , x_j, \ldots , x_n] = 0$ when $x_i=x_j$, the Leibniz $n$-algebra becomes an $n$-Lie  algebra.
In this case, the equality $(\ref{eq1})$ implies that $\q_{\Lie}^n=0$ and $Z_{\Lie}(\q)=\q$.
A Leibniz $n$-algebra $\q$ is called \Lie-abelian, if $\q_\Lie^n=0$ or equivalently $Z_\Lie(\q)=\q$.
Thus every $n$-Lie  algebra  is a \Lie-abelian Leibniz $n$-algebra.

Conversely, if  $\q$ is a \Lie-abelian Leibniz $2$-algebra
over a field $\mathbb{K}\ni \frac{1}{2}$, then $[x,x]_\Lie=0$ and so $[x,x]=0$. Hence $\q$
 is actually a Lie algebra. Therefore,  the Schur
\Lie-multiplier   \cite{c-i} and also the $c$-nilpotent Schur \Lie-multiplier  \cite{b-c3} of a Leibniz algebra
 are both Lie algebras, since they are \Lie-abelian.
  But the Schur \Lie-multiplier of a  Leibniz $n$-algebra $(n\geq3)$ is a \Lie-abelian Leibniz $n$-algebra which
  is not an $n$-Lie algebra, in general.
  In fact, for a Leibniz 2-algebra $\q$, we have $\q_\Lie^2=\langle [x,x]~ | ~x \in \q\rangle$, while
if $\q$ is a Leibniz $n$-algebra with $n\geq3$, then $\q_\Lie^n\subseteq {\sf _nLeib}$, where
 $${\sf _nLeib(\q)}=\left\langle [x_1,\ldots,x_i,\ldots,x_j,\ldots, x_n]~|~\exists i,j : x_i=x_j~\mbox{with}~ x_1,\ldots, x_n\in\q \right\rangle.$$
\end{remark}

 In the following example, we give a \Lie-abelian Leibniz $3$-algebra which is not an $n$-Lie algebra.

\begin{example}\label{ex1} \normalfont
Let  $\q=span\{x,y\}$ be the $2$-dimensional complex Leibniz $3$-algebra with non-zero multiplications $[x,y,y]=-2x$ and $[y,y,x]=[y,x,y]=x$.  It is easy to see that $Z_{\Lie}(\q)=\q$ and $\q_{\Lie}^3=0$.
\end{example}

Recall from \cite{s-b2} that an extension of  Leibniz $n$-algebras
$0\rightarrow\m\stackrel{\subseteq}{\rightarrow}\g\stackrel{\sigma}{\rightarrow}\q\rightarrow 0$
is said to be \Lie-central, if  $\m\subseteq Z_\Lie (\g)$ or equivalently $[\m,_{n-1}\g]_\Lie=0$. Moreover, a \Lie-central extension
   is called a \Lie-stem extension, whenever $\m\subseteq Z_\Lie (\g)\cap\g_\Lie^n$. In addition,
 a \Lie-stem extension  is said to be a \Lie-stem cover of $\q$, if
$\m\cong\mathcal{M}_\Lie(\q)$.
In this case, $\g$ is called a \Lie-cover of $\q$.

In \cite{s-b2}, it is discussed the following results on the Schur \Lie-multiplier of Leibniz $n$-algebras,
using their \Lie-central extensions.

\begin{lemma}\label{lem1}
Let $0\rightarrow\R\rightarrow\f\stackrel{\pi}{\rightarrow}\q\rightarrow 0$ be a free presentation of a Leibniz $n$-algebra
$\q$, and  $0\rightarrow\m\rightarrow\g\stackrel{\theta}{\rightarrow}\p\rightarrow 0$ be a \Lie-central extension of another Leibniz
$n$-algebra $\p$. Then for each homomorphism $\alpha:\q\to\p$, there exists a homomorphism
$\beta:\dfrac{\f}{[\R,_{n-1}\f]_\Lie}\to\g$ such that the following diagram is commutative:
\[\xymatrix{0\ar[r]&\dfrac{\R}{[\R,_{n-1}\f]_\Lie}\ar[r]\ar[d]_{\beta|}&\dfrac{\f}{[\R,_{n-1}\f]_\Lie}\ar[r]^{\ \ \ \ \ \bar{\pi}}\ar[d]_{\beta}& \q\ar[r]\ar[d]_{\alpha}&0\\
 0\ar[r]&\m\ar[r]& \g\ar[r]^{\theta}& \p\ar[r]&0,}\]
 where $\bar{\pi}$ is the natural epimorphism induced by $\pi$.
\end{lemma}

\begin{theorem}\label{th2}
Let $\q$ be a Leibniz $n$-algebra whose Schur \Lie-multiplier is finite dimensional and $\q\cong\f/\R$, where $\f$ is a free Leibniz $n$-algebra.
Then the extension $0\rightarrow\m\rightarrow\g\rightarrow\q\rightarrow 0$ is a \Lie-stem cover if and only if there exists an
ideal $\s$ of $\f$ such that
\begin{itemize}
\item[$(i)$] $\g\cong\f/\s$ and $\m\cong\R/\s$.
\item[$(ii)$] $\dfrac{\R}{[\R,_{n-1}\f]_\Lie}\cong\mathcal{M}_\Lie(\q)\oplus\dfrac{\s}{[\R,_{n-1}\f]_\Lie}$.
\end{itemize}
\end{theorem}

\begin{corollary}\label{coro3}
Any finite dimensional Leibniz $n$-algebra has at least one \Lie-cover.
\end{corollary}

\begin{corollary}\label{coro4}
All \Lie-stem covers of a Leibniz $n$-algebra   with finite dimensional Schur \Lie-multiplier are \Lie-isoclinic.
\end{corollary}

\begin{corollary}\label{coro5}
 Let $\q$ be  a Leibniz $n$-algebra with finite dimensional Schur \Lie-multiplier, and $\g$ be a \Lie-cover of $\q$. Then  $\g\sim\dfrac{\f}{[\R,_{n-1}\f]_\Lie}$, where $\f$ is a free Leibniz $n$-algebra such that $\q\cong\f/\R$.
\end{corollary}


\section{On \Lie-covers  of Leibniz $n$-algebras}\label{Lie-cover}
This section is devoted to obtain some properties of  \Lie-covers  of  Leibniz $n$-algebras.

\begin{lemma}\label{lem6}
Let $\q$ be a Leibniz $n$-algebra and
\[\xymatrix{0\ar[r]&\m'\ar[r]\ar[d]_{\alpha}&\g'\ar[r]\ar[d]_{\beta}& \q\ar[r]\ar[d]_{\gamma}&0\\
 0\ar[r]&\m\ar[r]& \g\ar[r]& \q\ar[r]&0,}\]
 be a commutative diagram of Leibniz $n$-algebras such that the first row is exact and the second
one is a \Lie-stem extension of $\q$. If the homomorphism $\gamma$ is onto, then so is $\beta$.
\end{lemma}
\begin{proof}
Clearly, $\g=\m +\Ima\beta$ and hence $\g_\Lie^n=[\m,_{n-1}\g]_\Lie + [\Ima\beta,_{n-1}\g]_\Lie$. Now, since $\m\subseteq Z_\Lie(\g)$,
we have  ${\g}_\Lie^n=[\Ima\beta,_{n-1}\g]_\Lie$, and since $\m\subseteq\g_\Lie^n$, we get $\m\subseteq [\Ima\beta,_{n-1}\g]_\Lie\subseteq\Ima\beta$.
Therefore, $\g=\Ima\beta$.
\end{proof}

\begin{proposition}\label{prop5}
Let
$0\rightarrow\m\rightarrow\g\stackrel{\sigma}{\rightarrow}\q\rightarrow 0$ be a \Lie-stem extension of a finite
dimensional Leibniz $n$-algebra $\q$. Then $\g$ is a homomorphic image of
a \Lie-cover of $\q$.
\end{proposition}
\begin{proof}
Let $0\rightarrow\R\rightarrow\f\stackrel{\pi}{\rightarrow}\q\rightarrow 0$ be a free presentation of $\q$. By Lemma \ref{lem1},
there exists a homomorphism
$\beta:\dfrac{\f}{[\R,_{n-1}\f]_\Lie}\to\g$ such that the following diagram is commutative:
\[\xymatrix{0\ar[r]&\dfrac{\R}{[\R,_{n-1}\f]_\Lie}\ar[r]\ar[d]_{\beta|}&\dfrac{\f}{[\R,_{n-1}\f]_\Lie}\ar[r]^{\ \ \ \ \ \bar{\pi}}\ar[d]_{\beta}& \q\ar[r]\ar[d]_{1_\q}&0\\
 0\ar[r]&\m\ar[r]& \g\ar[r]^{\sigma}& \q\ar[r]&0.}\]
 Lemma \ref{lem6} implies that $\beta$ is  an epimorphism and hence $\beta(\dfrac{\R}{[\R,_{n-1}\f]_\Lie})=\m$. Now, put $\ker\beta|=\ker\beta=\U/[\R,_{n-1}\f]_\Lie$,
 for some ideal $\U$ in $\R$. Thus $\g\cong\f/\U$ and $\m\cong\R/\U$. Clearly,
 \[\beta((\R\cap\f_\Lie^n)/[\R,_{n-1}\f]_\Lie)\subseteq\beta(\R/[\R,_{n-1}\f]_\Lie)\cap\beta(\f_\Lie^n/[\R,_{n-1}\f]_\Lie)=\m\cap\g_\Lie^n=\m.\]
 Conversely, let $y\in \beta(\R/[\R,_{n-1}\f]_\Lie)\cap\beta(\f_\Lie^n/[\R,_{n-1}\f]_\Lie)$.
 Then $y=\beta(r+[\R,_{n-1}\f]_\Lie)=\beta(x+[\R,_{n-1}\f]_\Lie)$, for some $r\in\R$ and $x\in \f_\Lie^n$. Therefore,
 $x-r+[\R,_{n-1}\f]_\Lie\in\ker\beta=\U/[\R,_{n-1}\f]_\Lie\subseteq\R/[\R,_{n-1}\f]_\Lie$. Hence, $x\in\R$ and
 $y\in \beta((\R\cap\f_\Lie^n)/[\R,_{n-1}\f]_\Lie)$. This shows that $\beta$ may be restricted to an epimorphism from
 $(\R\cap\f_\Lie^n)/[\R,_{n-1}\f]_\Lie$ onto $\m$. Thus
 \[\dfrac{\R}{\U}\cong\dfrac{\R/[\R,_{n-1}\f]_\Lie}{\U/[\R,_{n-1}\f]_\Lie}\cong\m\cong
 \dfrac{(\R\cap\f_\Lie^n)/[\R,_{n-1}\f]_\Lie}{(\U\cap\f_\Lie^n)/[\R,_{n-1}\f]_\Lie}\cong\dfrac{(\R\cap\f_\Lie^n)+\U}{\U},\]
 and by finite dimensionality, we have $\R=(\R\cap\f_\Lie^n)+\U$. Now, let $\s/[\R,_{n-1}\f]_\Lie$ be a complement of
 $(\U\cap\f_\Lie^n)/[\R,_{n-1}\f]_\Lie$ in $\U/[\R,_{n-1}\f]_\Lie$. Then $\s\cap(\R\cap\f_\Lie^n)=[\R,_{n-1}\f]_\Lie$
 and $\s+(\R\cap\f_\Lie^n)=\R$. Therefore,
 \[\dfrac{\R}{[\R,_{n-1}\f]_\Lie}\cong\mathcal{M}_\Lie(\q)\oplus\dfrac{\s}{[\R,_{n-1}\f]_\Lie}.\]
 Now, Theorem \ref{th2} implies that $\f/\s\cong\dfrac{\f/\U}{\s/\U}=\dfrac{\g}{\s/\U}$ is a \Lie-cover of $\q$.
\end{proof}

A Leibniz $n$-algebra $\q$ is said to be Hopfian, if every epimorphism $\varphi:\q\to\q$ is an
isomorphism (see also \cite{sal2}). Here, we give certain conditions under which any \Lie-cover of $\q$ is
a Hopfian Leibniz $n$-algebra.

\begin{proposition}\label{prop7}
Let $\q$ be a Leibniz $n$-algebra with finite dimensional Schur \Lie-multiplier, and
$0\rightarrow\m_i\rightarrow\g_i\rightarrow\q\rightarrow 0$ $(i=1,2)$ be two \Lie-stem covers  of $\q$.
If $\gamma:\g_1\to\g_2$ is an epimorphism such that $\gamma(\m_1)=\m_2$, then $\gamma$ is an isomorphism.
\end{proposition}
\begin{proof}
Let $0\rightarrow\R\rightarrow\f\stackrel{\pi}{\rightarrow}\q\rightarrow 0$ be a free presentation of $\q$. By Theorem \ref{th2},
there exist  ideals $\s_i$ $(i=1,2)$ of $\f$ such that  $\g_i\cong\f/\s_i$, $\m_i\cong\R/\s_i$
and $\R/[\R,_{n-1}\f]_\Lie\cong\mathcal{M}_\Lie(\q)\oplus\s_i/[\R,_{n-1}\f]_\Lie$.
Hence, one may consider the epimorphism  $\gamma:\f/\s_1\to\f/\s_2$ with $\gamma(\R/\s_1)=\R/\s_2$. Now, by the proof of Theorem \ref{th2},
there exists an epimorphism $\beta:\f/[\R,_{n-1}\f]_\Lie\to \f/\s_2$ such that $\ker\beta=\s_2/[\R,_{n-1}\f]_\Lie$.
Since $\f$ is free, there exists a homomorphism $\bar{\alpha}:\f\to\f/\s_1$ such that $\gamma\circ\bar{\alpha}=\beta\circ\pi'$,
where $\pi':\f\to{\f}/{[\R,_{n-1}\f]_\Lie}$ is the natural epimorphism. Clearly, $\bar{\alpha}$ induces a homomorphism
$\alpha:{\f}/{[\R,_{n-1}\f]_\Lie}\to{\f}/{\s_1}$ such that
the following diagram is commutative:
\[\xymatrix{0\ar[r]&\dfrac{\R}{[\R,_{n-1}\f]_\Lie} \ar[ddr]^{\beta_1}\ar[r]\ar[d]_{\alpha_1}&\dfrac{\f}{[\R,_{n-1}\f]_\Lie}\ar[ddr]^{\beta}
\ar[r]^{\ \ \ \ \ \bar{\pi}}\ar[d]_{\alpha}& \q\ar[ddr]\ar[r]\ar[d]_{\gamma_2^{-1}\circ\beta_2}&0\\
 0\ar[r]&\dfrac{\R}{\s_1}\ar[r]\ar[dr]_{\gamma_1}& \dfrac{\f}{\s_1}\ar[r]\ar[dr]_{\gamma}& \q\ar[r]^{\ \ \ \beta_2}
 \ar[dr]_{\gamma_2}&0\\
 &0\ar[r]&\dfrac{\R}{\s_2}\ar[r]& \dfrac{\f}{\s_2}\ar[r]& \q\ar[r]&0,}\]
where $\alpha_1$, $\beta_1$ and $\gamma_1$ are the restriction homomorphisms of $\alpha$, $\beta$ and $\gamma$, respectively, and
also $\beta_2$ and $\gamma_2$ are the induced isomorphisms of $\beta$ and $\gamma$, respectively. Now, Lemma \ref{lem6}
implies that $\alpha$ is onto. Put $\ker\alpha=\U/[\R,_{n-1}\f]_\Lie$, for some ideal $\U$ of $\f$. Then
$\f/\U\cong\f/\s_1$. On the other hand, since $\ker\alpha\subseteq\ker\beta$, we have
$\U\subseteq\s_2$, and since  $\U+(\R\cap\f_\Lie^n)=\R$, we get $\U=\s_2$, which shows that $\gamma$ is an isomorphism.
\end{proof}

\begin{corollary}\label{coro8}
If $\q$ is a \Lie-abelian Leibniz $n$-algebra with finite dimensional Schur \Lie-multiplier, then every \Lie-cover of $\q$ is
Hopfian.
\end{corollary}
\begin{proof}
Suppose that $0\rightarrow\m\rightarrow\g\rightarrow\q\rightarrow 0$ is a \Lie-stem cover of $\q$. Since
$\dfrac{\g_\Lie^n}{\m}=\big{(}\dfrac{\g}{\m}\big{)}_\Lie^n=0$, we have $\g_\Lie^n=\m$.
Now, if $\gamma:\g\to\g$ is an epimorphism, then $\gamma(\m)=\gamma(\g_\Lie^n)=\g_\Lie^n=\m$
and  hence by the above theorem, $\gamma$ is an isomorphism.
\end{proof}

\begin{proposition}\label{prop9}
Let $\q$ be a Leibniz $n$-algebra with finite dimensional Schur \Lie-multiplier, and
$0\rightarrow\m_i\rightarrow\g_i\rightarrow\q\rightarrow 0$ $(i=1,2)$ be two \Lie-stem covers  of $\q$.
Then $Z_\Lie(\g_1)/\m_1\cong Z_\Lie(\g_2)/\m_2$.
\end{proposition}
\begin{proof}
Let $0\rightarrow\R\rightarrow\f\rightarrow\q\rightarrow 0$ and $0\rightarrow\m\rightarrow\g\rightarrow\q\rightarrow 0$
 be a free presentation and a \Lie-stem cover of $\q$, respectively.
  Theorem \ref{th2} implies that
there exists  an ideal $\s$  of $\f$ such that  $\g\cong\f/\s$, $\m\cong\R/\s$
and $\R/[\R,_{n-1}\f]_\Lie\cong\mathcal{M}_\Lie(\q)\oplus\s/[\R,_{n-1}\f]_\Lie$.
 Now, put $Z_\Lie(\f/[\R,_{n-1}\f]_\Lie)=\U/[\R,_{n-1}\f]_\Lie$, for some ideal $\U$ of $\f$. Clearly, $[\U,_{n-1}\f]_\Lie\subseteq[\R,_{n-1}\f]_\Lie$
 and hence $\U/\s\subseteq Z_\Lie(\f/\s)$. Conversely, let $x+\s\in Z_\Lie(\f/\s)$. Then
  $[x,_{n-1}\f]_\Lie\subseteq \s\cap\f_\Lie^n=[\R,_{n-1}\f]_\Lie$, which implies that $x+[\R,_{n-1}\f]_\Lie\in Z_\Lie(\f/[\R,_{n-1}\f]_\Lie)$.
  Therefore, $x+\s\in\U/\s$ and hence  $Z_\Lie(\f/\s)=\U/\s$. Thus
  $Z_\Lie(\g)/\m\cong \U/\R$, which completes the proof.
\end{proof}

In what follows, we obtain some results on \Lie-covers of a \Lie-perfect Leibniz $n$-algebra. A Leibniz $n$-algebra $\q$ is said to be \Lie-perfect,
whenever $\q=\q_\Lie^n$.

Let $G_i:0\rightarrow\m_i\rightarrow\g_i\stackrel{\sigma_i}{\rightarrow}\q\rightarrow 0$ $(i=1,2)$ be two \Lie-central extensions  a  Leibniz $n$-algebra $\q$. Recall from \cite{s-b2} that $(\beta|,\beta,1_\q):G_1\longrightarrow G_2$ is  a homomorphism of \Lie-central extensions, if the following diagram of homomorphisms is commutative:
\[\xymatrix{G_1:0\ar[r]&\m_1\ar[r]\ar[d]_{\beta|}&\g_1\ar[r]^{\sigma_1}\ar[d]_{\beta}& \q\ar[r]\ar[d]_{1_\q}&0\\
G_2: 0\ar[r]&\m_2\ar[r]& \g_2\ar[r]^{\sigma_2}& \q\ar[r]&0.
}\]
 Moreover, the \Lie-central extension $G_1$ is said to be universal, if
 for every  \Lie-central extension $G_2$, there exists a unique homomorphism $(\beta|,\beta,1_\q):G_1\longrightarrow G_2$.\\

The following result shows that every \Lie-perfect Leibniz $n$-algebra admits at least one \Lie-cover.

\begin{theorem}\label{th10}
Let $\q$ be a \Lie-perfect Leibniz $n$-algebra with a free presentation $\q\cong\f/\R$. Then
 the \Lie-central extension
 \begin{equation}\label{e1}
 0\longrightarrow\mathcal{M}_\Lie(\q)\longrightarrow\dfrac{\f_\Lie^n}{[\R,_{n-1}\f]_\Lie}\stackrel{\rho}{\longrightarrow}\q\longrightarrow 0
 \end{equation}
    where $\rho(x+[\R,_{n-1}\f]_\Lie)=x+\R$, is a \Lie-stem cover of $\q$. Moreover, it is a universal \Lie-central extension.
\end{theorem}
\begin{proof}
Since $\q$ is \Lie-perfect, we have $\f=\f_\Lie^n+\R$  and hence $\f_\Lie^n=(\f_\Lie^n+\R)_\Lie^n=(\f_\Lie^n)_\Lie^n+[\R,_{n-1}\f]_\Lie$. Thus
 $$\Big{(}\dfrac{\f_\Lie^n}{[\R,_{n-1}\f]_\Lie}\Big{)}_\Lie^n =  \dfrac{\f_\Lie^n}{[\R,_{n-1}\f]_\Lie},$$
which actually shows that $\f_\Lie^n/[\R,_{n-1}\f]_\Lie$ is a \Lie-perfect Leibniz $n$-algebra. Therefore, the \Lie-central extension (\ref{e1}) is a \Lie-stem cover of $\q$.

Now, let $G:0\rightarrow\m\rightarrow\g\stackrel{\sigma}{\rightarrow}\q\rightarrow 0$ be an arbitrary \Lie-central extension of $\q$.
By Lemma \ref{lem1}, we have the following commutative diagram:
\[\xymatrix{F:0\ar[r]&\mathcal{M}_\Lie(\q)\ar[r]\ar[d]&\dfrac{\f_\Lie^n}{[\R,_{n-1}\f]_\Lie}\ar[r]^{\ \ \ \ \ \rho}\ar[d]_{\subseteq}& \q\ar[r]\ar[d]_{1_\q}&0\\
\ \ \ \ \  0\ar[r]&\dfrac{\R}{[\R,_{n-1}\f]_\Lie}\ar[r]\ar[d]_{\beta|}& \dfrac{\f}{[\R,_{n-1}\f]_\Lie}\ar[r]^{\ \ \ \ \ \bar{\pi}}\ar[d]_{\beta}& \q\ar[r]\ar[d]_{1_\q}&0\\
G: 0\ar[r]&\m\ar[r]& \g\ar[r]^{\sigma}& \q\ar[r]&0,}\]
which shows that there exists a homomorphism from $F$ to $G$. We show that this homomorphism is unique. So, let
$\varphi_i:\f_\Lie^n/[\R,_{n-1}\f]_\Lie\to\g$ $(i=1,2)$ be two homomorphisms such that the following diagram is commutative:
\[\xymatrix{F:0\ar[r]&\mathcal{M}_\Lie(\q)\ar[r]\ar[d]_{\varphi_i|}&\dfrac{\f_\Lie^n}{[\R,_{n-1}\f]_\Lie}\ar[r]^{\ \ \ \ \ \rho}\ar[d]_{\varphi_i}& \q\ar[r]\ar[d]_{1_\q}&0\\
G: 0\ar[r]&\m\ar[r]& \g\ar[r]^{\sigma}& \q\ar[r]&0.}\]
Therefore, $\sigma\circ\varphi_1=\rho=\sigma\circ\varphi_2$ and hence $\varphi_1(x)-\varphi_2(x)\in\ker\sigma\subseteq Z_\Lie(\g)$, for all
$x\in\f_\Lie^n/[\R,_{n-1}\f]_\Lie$. Now, we prove that $\varphi_1([x_1,\ldots,x_n]_\Lie)=\varphi_2([x_1,\ldots,x_n]_\Lie)$, for all
$x_i\in\f_\Lie^n/[\R,_{n-1}\f]_\Lie$ and since $\f_\Lie^n/[\R,_{n-1}\f]_\Lie$ is \Lie-perfect, we get $\varphi_1=\varphi_2$, which completes the proof.
Clearly,
\begin{equation*}
[\varphi_1(x_1)-\varphi_2(x_1),\varphi_1(x_2),\ldots,\varphi_1(x_n)]_\Lie =0,
\end{equation*}
and so
\begin{equation*}
\varphi_1([x_1,\ldots,x_n]_\Lie)=[\varphi_2(x_1),\varphi_1(x_2),\ldots,\varphi_1(x_n)]_\Lie=\ast.
\end{equation*}
On the other hand, since
\begin{equation*}
[\varphi_2(x_1),\varphi_1(x_2)-\varphi_2(x_2),\varphi_1(x_3),\ldots,\varphi_1(x_n)]_\Lie =0,
\end{equation*}
 we have
\begin{equation*}
\ast=[\varphi_2(x_1),\varphi_2(x_2),\varphi_1(x_3),\ldots,\varphi_1(x_n)]_\Lie.
\end{equation*}
 By repeating the above process, we finally obtain
  \begin{equation*}
\ast=[\varphi_2(x_1),\varphi_2(x_2),\ldots,\varphi_2(x_n)]_\Lie=\varphi_2([x_1,\ldots,x_n]_\Lie).
\end{equation*}
\end{proof}

In \cite{l2}, Loday discussed a necessary and sufficient condition for a pair of groups to be perfect,
using a homological method.
In the next theorem, we show this result for
a Leibniz $n$-algebra.

\begin{theorem}\label{th11}
A Leibniz $n$-algebra $\q$ is \Lie-perfect if and only if it has a universal \Lie-central extension.
\end{theorem}
\begin{proof}
If $\q$ is \Lie-perfect, then  the \Lie-central extension (\ref{e1}) in  Theorem \ref{th10} is universal.
Conversely, let $G:0\rightarrow\m\rightarrow\g\stackrel{\sigma}{\rightarrow}\q\rightarrow 0$ be
a universal \Lie-central extension of $\q$.
Clearly,
\[0\longrightarrow\m\times\dfrac{\g}{\g_\Lie^n}\longrightarrow\g\times\dfrac{\g}{\g_\Lie^n}\stackrel{\delta}{\longrightarrow}\q\longrightarrow 0,\]
where $\delta(x,y+\g_\Lie^n)=\sigma(x)$, is a  \Lie-central extension of $\q$. Define homomorphisms
$\psi_i:\g\longrightarrow\g\times\g/\g_\Lie^n$ ($i=1,2$) by $\psi_1(x)=(x,0)$ and $\psi_2(x)=(x,\sigma(x)+\g_\Lie^n)$. It is easy to see that
the following diagram is commutative ($i=1,2$):
\[\xymatrix{G:0\ar[r]&\m\ar[r]\ar[d]_{\psi_i|}&\g\ar[r]^{\sigma}\ar[d]_{\psi_i}& \q\ar[r]\ar[d]_{1_\q}&0\\
\ \ \ \ \ 0\ar[r]&\m\times\dfrac{\g}{\g_\Lie^n}\ar[r]& \g\times\dfrac{\g}{\g_\Lie^n}\ar[r]^{\ \ \ \ \delta}& \q\ar[r]&0.}\]
Then the universal property of $G$ implies that $\psi_1=\psi_2$. Therefore, $\g$ is \Lie-perfect and since $\sigma$ is onto, $\q$
is a \Lie-perfect Leibniz $n$-algebra, as well.
\end{proof}

The next result easily follows from Theorem \ref{th10} and Corollary \ref{coro5}.

\begin{corollary}\label{coro12}
Let $\q$ be a \Lie-perfect Leibniz $n$-algebra with finite dimensional Schur \Lie-multiplier,
and $\g$ be a \Lie-cover of $\q$. Then  $\g\sim\dfrac{\f_\Lie^n}{[\R,_{n-1}\f]_\Lie}\sim\dfrac{\f}{[\R,_{n-1}\f]_\Lie}$, where $\f$ is a free Leibniz $n$-algebra such that $\q\cong\f/\R$.
\end{corollary}

It is well-known that isomorphism  and isoclinism are
equivalent for finite dimensional stem Lie algebras and stem  $n$-Lie algebras (see \cite{m,e-m-s}).
But this is not true for finite dimensional $\Lie$-stem Leibniz $n$-algebras ($n\geq 2$) (see \cite{r-c}). Moreover in \cite[Corollary 3]{r-c}, it is discussed certain conditions under which  these concepts are equivalent.

\begin{lemma}\label{lem13}
Let $\q$ and $\p$ be two finite dimensional \Lie-isoclinic \Lie-stem complex Leibniz
$2$-algebras and for all elements $x_1,x_2\in\q$, there exists $\varepsilon_{12}\in\mathbb{C}$ such that $[x_1,x_2]=\varepsilon_{12} [x_2,x_1]$.
Then  $\q$ and $\p$ are isomorphic.
\end{lemma}

Note that in the proof of above result,  authors use the fact that $Z_{\Lie}(\q)= Z(\q)$, for any $\Lie$-stem Leibniz $2$-algebra $\q$,
 where $Z(\q)=\{z\in \q|\ [x,z]=[z,x]=0, \forall x\in\q\}$ is the center of $\q$
(see \cite[Lemma 4]{r-c}). But this is not true for \Lie-stem Leibniz $n$-algebras for $n\geq3$.
 In \cite[Example 4.5]{s-b}, it is given a  \Lie-stem complex Leibiniz 3-algebra whose \Lie-center
is not central. Hence, we discuss the following result for Leibniz 2-algebras.

\begin{corollary}\label{coro14}
Let $\q$ be a \Lie-perfect complex Leibniz $2$-algebra with finite dimensional Schur \Lie-multiplier
with a free presentation $\q\cong\f/\R$, in which for all elements
$x_1,x_2\in[\f,\f]_\Lie/[\R,\f]_\Lie$ there exists $\varepsilon_{12}\in\mathbb{C}$ such that $[x_1,x_2]=\varepsilon_{12} [x_2,x_1]$.
 If $\g$ is a \Lie-cover of $\q$, then  $\g\cong[\f,\f]_\Lie/[\R,\f]_\Lie$.
\end{corollary}
\begin{proof}
Clearly, every \Lie-cover of  $\q$ is \Lie-perfect. Now, Corollary \ref{coro12} and Lemma \ref{lem13}  complete the proof.
\end{proof}


\section{On dimension of the Schur \Lie-multiplier of a Leibniz $n$-algebra}
In this section, we give some equalities and inequalities for the dimension of the Schur \Lie-multiplier of a Leibniz $n$-algebra.

\begin{proposition}\label{prop15}\cite[Proposition 4.1]{s-b2}
Let $0 \to \R\to \f \stackrel{\pi} \to \q\to 0$ be a free presentation of a Leibniz $n$-algebra $\q$. Also,
let $\m$ be an ideal of $\q$ and $\s$ an ideal of $\f$ such that $\m\cong\s/\R$. Then the following sequences are exact:
\begin{itemize}
\item[$(i)$] $0 \to \dfrac{\R\cap [\s,_{n-1}\f]_\Lie}{[\R,_{n-1}\f]_\Lie}\to  \mathcal{M}_\Lie(\q) \to \mathcal{M}_\Lie(\dfrac{\q}{\m}) \to \dfrac{\m \cap\q_\Lie^n}{[\m,_{n-1}\q]_\Lie} \to 0$,
\item[$(ii)$] $\mathcal{M}_\Lie(\q) \to \mathcal{M}_\Lie(\dfrac{\q}{\m}) \to \dfrac{\m \cap\q_\Lie^n}{[\m,_{n-1}\q]_\Lie} \to\dfrac{\m}{[\m,_{n-1}\q]_\Lie}
\to\dfrac{\q}{\q_\Lie^n}\to \dfrac{\q}{\m+\q_\Lie^n}\to 0$.
\end{itemize}
\end{proposition}

\begin{corollary}\label{coro16}
Under the notation of the above result, if $\m\subseteq Z_\Lie(\q)$ then the following sequence is exact:
\[\m\otimes^{n-1}\dfrac{\q}{_n{\Leib(\q)}}\to  \mathcal{M}_\Lie(\q) \to \mathcal{M}_\Lie(\dfrac{\q}{\m}) \to \m \cap\q_\Lie^n  \to 0,\]
where   $_n{\sf Leib(\q)}=\left\langle [x_1,\ldots,x_i,\ldots,x_j,\ldots, x_n]~|~\exists i,j : x_i=x_j~\mbox{with}~ x_1,\ldots, x_n\in\q \right\rangle.$
\end{corollary}
\begin{proof}
Since  $\m\subseteq Z_\Lie(\q)$, we have $[\s,_{n-1}\f]_\Lie\subseteq\R$. Then
by Proposition \ref{prop15} (i), the following  sequence is exact:
\[0 \to \dfrac{[\s,_{n-1}\f]_\Lie}{[\R,_{n-1}\f]_\Lie}\to  \mathcal{M}_\Lie(\q) \to \mathcal{M}_\Lie(\dfrac{\q}{\m}) \to \m \cap\q_\Lie^n \to 0.\]
Now, define
\begin{equation}
\begin{array}{rcl}
\varphi:\m\otimes^{n-1}\dfrac{\q}{_n\Leib(\q)}&\longrightarrow & \dfrac{\R\cap\f_\Lie^n}{[\R,_{n-1}\f]_\Lie}=\mathcal{M}_\Lie(\q)\\ \\ \nonumber
m\otimes \bar{x}_1\otimes\cdots\otimes\bar{x}_{n-1}&\longmapsto & [s,f_1,\ldots,f_{n-1}]_\Lie+[\R,_{n-1}\f]_\Lie,
\end{array}
\end{equation}
where $\pi(s)=m$ and $\pi(f_i)=x_i$, for all $1\leq i\leq n-1$. Since $\m$ is \Lie-central,  $\varphi$ is a well-defined homomorphism. Clearly,
$\Ima\varphi={[\s,_{n-1}\f]_\Lie}/{[\R,_{n-1}\f]_\Lie}$, which completes the proof.
\end{proof}
Note that the  $n$-Lie algebra $\q/_n{\Leib(\q)}$ is actually a \Lie-abelian Leibniz $n$-algebra (see Remark \ref{rem17}).
The next result easily follows from  Proposition \ref{prop15}  and the following commutative diagram.
\begin{equation*}
\xymatrix{
 \R \cap[\s,_{n-1}\f]_\Lie \ar[r] \ar[d] & \R \cap \f_\Lie^n \ar[r] \ar[d] & \dfrac{\R \cap\f_\Lie^n}{\R \cap [\s,_{n-1}\f]_\Lie} \ar[d]\\
[\s,_{n-1}\f]_\Lie \ar[r] \ar[d] & \s \cap \f_\Lie^n \ar[r] \ar[d] & \dfrac{\s \cap \f_\Lie^n}{[\s,_{n-1}\f]_\Lie} \ar[d]\\
 [\m,_{n-1}\q]_\Lie \cong \dfrac{[\s,_{n-1}\f]_\Lie}{\R \cap [\s,_{n-1}\f]_\Lie} \ar[r]& \m \cap \q_\Lie^n \ar[r] &
 \dfrac{\m \cap  \q_\Lie^n}{[\m,_{n-1}\q]_\Lie}}
\end{equation*}

\begin{corollary}\label{coro17}
Under the notation of Proposition \ref{prop15}, if $\q$ is a finite dimensional Leibniz $n$-algebra, then
\begin{itemize}
\item[$(i)$] $\mathcal{M}_\Lie(\q)$ is finite dimensional,
\item[$(ii)$] $\dim\mathcal{M}_\Lie(\dfrac{\q}{\m}) \leq \dim\mathcal{M}_\Lie(\q) + \dim\dfrac{\m \cap\q_\Lie^n}{[\m,_{n-1}\q]_\Lie}$,
\item[$(iii)$] $\dim\mathcal{M}_\Lie(\q) + \dim(\m \cap\q_\Lie^n) = \dim\mathcal{M}_\Lie(\dfrac{\q}{\m}) + \dim[\m,_{n-1}\q]_\Lie + \dim\dfrac{\R\cap [\s,_{n-1}\f]_\Lie}{[\R,_{n-1}\f]_\Lie}$,
\item[$(iv)$]  $\dim\mathcal{M}_\Lie(\q) + \dim(\m \cap\q_\Lie^n) = \dim\mathcal{M}_\Lie(\dfrac{\q}{\m}) +
\dim\dfrac{[\s,_{n-1}\f]_\Lie}{[\R,_{n-1}\f]_\Lie}$,
\item[$(v)$] $\dim\mathcal{M}_\Lie(\q) + \dim\q_\Lie^n = \dim\dfrac{\f_\Lie^n}{[\R,_{n-1}\f]_\Lie}$,
\item[$(vi)$] if $\mathcal{M}_\Lie(\q)=0$, then $\mathcal{M}_\Lie(\dfrac{\q}{\m})\cong\dfrac{\m \cap\q_\Lie^n}{[\m,_{n-1}\q]_\Lie}$,
\item[$(vii)$] if $\m$ is a \Lie-central ideal of $\q$, then
\begin{eqnarray*}
\dim\mathcal{M}_\Lie(\q) + \dim(\m \cap\q_\Lie^n) &\leq& \dim\mathcal{M}_\Lie(\dfrac{\q}{\m}) +
\dim\big{(}\m\otimes^{n-1}\dfrac{\q}{_n{\Leib(\q)}}\big{)}\\
&\leq&\dim\mathcal{M}_\Lie(\dfrac{\q}{\m}) +
\dim\big{(}\m\otimes^{n-1}\dfrac{\q}{\q_\Lie^n}\big{)}.
\end{eqnarray*}
\end{itemize}
\end{corollary}

Let $\q$ be a Leibniz $n$-algebra and $\m$ be an ideal of $\q$. Then $(\m,\q)$ is called a pair of  Leibniz $n$-algebras (see
 \cite{b-s,r-c}). Suppose that $\q\cong\f/\R$ is a free presentation of $\q$ such that $\m\cong\s/\R$, for an ideal
 $\s$ of $\f$. Then one may define the Schur \Lie-multiplier of a pair
 of Leibniz $n$-algebras as
 \[\mathcal{M}_\Lie(\m,\q)=\dfrac{\R\cap [\s,_{n-1}\f]_\Lie}{[\R,_{n-1}\f]_\Lie},\]
 which is actually the first non-zero term in the exact sequence given in Proposition \ref{prop15} (i).
 Clearly if $\m=\q$, then this definition coincides with  the  Schur \Lie-multiplier of a Leibniz $n$-algebra
 discussed in this paper, and
 in addition if $n=2$, then it yields the  Schur \Lie-multiplier of a Leibniz algebra given in \cite{c-i}.

The next result follows from Corollary \ref{coro17} (iii) and \cite[Lemma 2.6 (iii)]{s-b}.

\begin{corollary}\label{coro18}
Let $\q$ be a finite dimensional Leibniz $n$-algebra and $\m$ be an ideal of $\q$ such that $\q$ is \Lie-isoclinic to $\q/\m$.
Then
\[\dim\mathcal{M}_\Lie(\q)  = \dim\mathcal{M}_\Lie(\dfrac{\q}{\m}) + \dim\mathcal{M}_\Lie(\m,\q).\]
\end{corollary}


In what follows, we obtain  upper bounds for the dimension of the \Lie-commutator of a Leibniz $n$-algebra with finite dimensional
\Lie-central factor as well as  for the dimension of the Schur \Lie-multiplier of a finite dimensional Leibniz $n$-algebra.

In 1904, Schur  \cite{s} proved  that if   the central factor of a group $G$ is finite, then so is
  $G'$, where $G'$ is the commutator subgroup of $G$.
Also,  Wiegold \cite{w}
 showed that if  $|G/Z(G)|=p^k$, then $G'$ is a $p$-group of order at most $p^{\frac{1}{2}k(k-1)}$.

Nearly a century after Schur, Moneyhun \cite{m}  proved  that if $L$ is a Lie algebra with $\dim L/Z(L)=k$, then $\dim [L,L]\leq \frac{1}{2}k(k-1)$.
 In fact, if $\{\bar{x}_1,\ldots,\bar{x}_k\}$ is a basis for $L/Z(L)$, then $[L,L]$
 can be generated by
 $\{[x_{i},x_{j}]:\ 1\leq i< j\leq k\}$.
  Therefore,  $\dim [L,L]\leq {k\choose2}$.
  Furthermore, if $L$ is an $n$-Lie algebra  such that $\dim L/Z(L)=k$, then  one can similarly show that
  the dimension of the commutator  of $L$ is at most  ${k\choose n}$ (see \cite{d-s}).

The  structure of a group (resp. Lie algebra) and its central factor, with respect to the  order (resp. dimension) of  its commutator has been already studied by many authors (see \cite{a-p-s,d-s,i,p-z}, for instance). Now, we obtain an upper bound for the dimension of the \Lie-commutator of a Leibniz $n$-algebra with finite
dimensional \Lie-central factor.

\begin{theorem}\label{th19}
Let $\q$ be a Leibniz $n$-algebra such that $\dim(\q/Z_\Lie(\q))=k$. Then
\[\dim\q_\Lie^n\leq \sum_{i=1}^{n} {{n-1}\choose{i-1}}{k\choose i}.\]
\end{theorem}
\begin{proof}
Let $\{\bar{x}_1,\ldots,\bar{x}_k\}$ be a basis for $\q/Z_\Lie(\q)$. We should actually find the cardinal number of the set
\[B=\{[x_{j_1},\ldots,x_{j_n}]_\Lie:\ 1\leq j_1\leq\cdots\leq j_n\leq k\}.\]
In fact,
 $$\dim\q_\Lie^n\leq \Gamma_1+\Gamma_2+\cdots+\Gamma_{n},$$
 where $\Gamma_i$ is the number of elements $[x_{j_1},\ldots,x_{j_n}]_\Lie$ in $B$ such that the set $\{x_{j_1},\ldots,x_{j_n}\}$ contains
 exactly $i$ distinct members
 ($1\leq i\leq n$).
Therefore, $\Gamma_1=k={k\choose 1}$, since we have these $k$ elements: $$[x_1,\ldots,x_1]_\Lie,\ldots, [x_k,\ldots,x_k]_\Lie.$$
 Also  $\Gamma_2={n-1\choose 1}{k\choose 2}$, because of elements of the form:
 \[[x_{j_s},x_{j_t},\ldots,x_{j_t}]_\Lie, [x_{j_s},x_{j_s},x_{j_t}\ldots,x_{j_t}]_\Lie,\ldots,[x_{j_s},\ldots,x_{j_s},x_{j_t}]_\Lie,\]
 in which $1\leq j_s < j_t\leq k$.
 By a similar combinatorial computation, one can deduce that
 $$\Gamma_i={{n-1}\choose{i-1}}{k\choose i},$$
 which completes the proof.\\
\end{proof}

\begin{corollary}\label{coro20}
Let $\q$ be a Leibniz $2$-algebra such that $\dim(\q/Z_\Lie(\q))=k$. Then
\[\dim[\q,\q]_\Lie\leq \frac{1}{2}k(k+1).\]
\end{corollary}

\begin{remark}\label{rem21} \normalfont
Let $\q$ be a Leibniz $n$-algebra,
\[Z(\q)=\{x\in\q|\ [x,\q,\ldots,\q]=\cdots=[\q,\ldots,\q,x]=0\}\]
be the center, and $\q^n=[\q,\ldots,\q]$ be the commutator of $\q$. It is easy to see that if $\dim(\q/Z(\q))=k$, then
$\dim\q^n\leq k^n$.
\end{remark}

In \cite{m}, Moneyhun proved  that if $L$ is a Lie algebra of dimension $k$, then $\dim\mathcal{M}(L)\leq \frac{1}{2}k(k-1)$.
Also, for   a $k$-dimensional $n$-Lie algebra $L$, we have $\dim\mathcal{M}(L)\leq{k\choose n}$ (see \cite{d-s}).
In the next result, we give an upper bound for the dimension of the Schur \Lie-multiplier of a finite dimensional Leibniz $n$-algebra.

\begin{theorem}\label{coro22}
Let $\q$ be a $k$-dimensional Leibniz $n$-algebra. Then
\[\dim\mathcal{M}_\Lie(\q)\leq\sum_{i=1}^{n} {{n-1}\choose{i-1}}{k\choose i}.\]
In particular, if the equality occurs, then $\q$ is a \Lie-abelian Leibniz $n$-algebra.
\end{theorem}
\begin{proof}
Let $0\rightarrow\m\rightarrow\g\rightarrow\q\rightarrow 0$ be a \Lie-stem cover of $\q$. Clearly,
$\dim(\g/Z_\Lie(\g))\leq\dim(\g/\m) =k$, thus  $\dim\g_\Lie^n\leq\sum_{i=1}^{n} {{n-1}\choose{i-1}}{k\choose i}$, thanks to Theorem \ref{th19}.
Hence
\[\dim\mathcal{M}_\Lie(\q)=\dim\m\leq\dim\g_\Lie^n\leq\sum_{i=1}^{n} {{n-1}\choose{i-1}}{k\choose i}.\]
Now, if  the equality holds, then the above inequality implies that
$\m=\g_\Lie^n$ and hence $\q=\g/\m$ is  \Lie-abelian.
\end{proof}

In \cite{d-s}, it is shown that a $k$-dimensional $n$-Lie algebra $L$ is abelian if and only if $\dim\mathcal{M}(L)={k\choose n}$.
But the following example shows that the converse of the last statement of Theorem \ref{coro22} is not true when $n\geq 3$.

\begin{example}\label{ex23} \normalfont
Let $\q=span\{x,y\}$ be the $2$-dimensional complex Leibniz $3$-algebra with non-zero multiplications
$[x,x,y]=-[y,x,x]=y$. Clearly, $\q_\Lie^3=0$ and so $\q$ is \Lie-abelian. Moreover,
$$\dim\dfrac{\q}{_3{\Leib(\q)}}=\dim\dfrac{\q}{span\{y\}}=1.$$
Now, in Corollary \ref{coro17} (vii), if we put $\m=\q$, then
$$\dim\mathcal{M}_\Lie(\q)\leq\dim(\q\otimes^2\dfrac{\q}{_3{\Leib(\q)}})\leq\dim\q=2,$$
while $\sum_{i=1}^{3} {{2}\choose{i-1}}{2\choose i}=4$. Also, Example \ref{ex1} is another counter-example.
\end{example}

\begin{corollary}\label{coro24}
Let $\q$ be a  $k$-dimensional Leibniz $2$-algebra. Then
\[\dim\mathcal{M}_\Lie(\q)\leq \frac{1}{2}k(k+1),\]
and  the equality holds if and only if $\q$ is \Lie-abelian.
\end{corollary}
\begin{proof}
If the equality holds then by Theorem \ref{coro22}, $\q$ is \Lie-abelian. Conversely,
let $\q$ be a \Lie-abelian Leibniz $2$-algebra. Then for every \Lie-stem extension $0\rightarrow\m\rightarrow\g\rightarrow\q\rightarrow 0$
 of $\q$, we have $\m=\g_\Lie^2\subseteq Z_\Lie(\g)$. Similar to the Lie case given in \cite[Lemma 23]{m}, assume that $\mathfrak{u}$ and $\mathfrak{v}$ are vector spaces
with bases $A=\{x_1,\ldots,x_k\}$ and  $B=\{y_{ij}|\ 1\leq i\leq j\leq k\}$, respectively, and put
$\g=\mathfrak{u}+\mathfrak{v}$ and consider the multiplications: $[x_i,x_j]=y_{ij}$ for $1\leq i\leq j\leq k$,
$[x_j,x_i]=0$  for $1\leq i< j\leq k$ and $[y_{ij},g]=[g,y_{ij}]=0$ for every $g\in \g$.
Therefore, $\g$ is a Leibniz algebra (by linear extending the above products  to all of $\g$). Moreover
if $i<j$, then $y_{ij}=[x_i,x_j]=[x_i,x_j]_\Lie\in \g_\Lie^2$ and also $y_{ii}=[x_i,x_i]=\frac{1}{2}[x_i,x_i]_\Lie\in\g_\Lie^2$. Hence
  $\g_\Lie^2=\mathfrak{v}$.
Furthermore, $\dim(\g/\g_\Lie^2)=\dim\mathfrak{u}=k$ and $0\to\g_\Lie^2\to\g\to\q\to 0$ is a \Lie-stem extension of $\q$ in which $\g$ is of maximal dimension.
Now, Proposition \ref{prop5} implies that $\g$ is a \Lie-cover of $\q$ and so $\dim\mathcal{M}_\Lie(\q)=\dim \g_\Lie^2=|B|=\frac{1}{2}k(k+1)$.
\end{proof}

\begin{corollary}\label{coro25}
If $\q$ is a  $k$-dimensional Leibniz $n$-algebra, then
\[\dim\mathcal{M}_\Lie(\q)\leq \sum_{i=1}^{n} {{n-1}\choose{i-1}}{k\choose i} -\dim\q_\Lie^n.\]
\end{corollary}
\begin{proof}
Let $0\to\R\to\f\to\q\to0$ be a free presentation of $\q$. Since
\[\dim\big{(}\frac{\f/[\R,_{n-1}\f]_\Lie}{Z_\Lie(\f/[\R,_{n-1}\f]_\Lie)}\big{)}\leq\dim(\f/\R)=k,\]
by Theorem \ref{th19} we get $\dim(\f_\Lie^n/[\R,_{n-1}\f]_\Lie)\leq \sum_{i=1}^{n} {{n-1}\choose{i-1}}{k\choose i}$.
Now the result follows from Corollary \ref{coro17} (v).
\end{proof}

\begin{remark}\label{prob8} \normalfont
Let $L$ be a  Lie algebra of dimension $k$. Then $t(L)=\frac{1}{2}k(k-1)-\dim\mathcal{M}(L)$ is a non-negative integer.
One of the most interesting problem in the study of Lie algebras  is  classifying  Lie algebras by $t(L)$ (see \cite{n}).
 For example in \cite{b-m-s}, Batten et al. characterized nilpotent Lie algebras with $t(L)=0,1,2$, and
later Hardy and Stitzinger \cite{h-s} when  $t(L)=3,4,5,6$ and then  for $t(L)=7,8$ in \cite{h}.
 A similar problem arises in the case of Lie superalgebras  \cite{n1,safa}, and also in finite group theory  \cite{brk, lls, gsc, zh}.

Here, the  same question appears for Leibniz $n$-algebras. The characterizing of $k$-dimensional Leibniz $n$-algebras by $t(n,\q)$, where
\[t(n,\q)=\sum_{i=1}^{n} {{n-1}\choose{i-1}}{k\choose i}-\dim\mathcal{M}_\Lie(\q).\]
\end{remark}


Hesam Safa \\
Department of Mathematics, Faculty of Basic Sciences, University of Bojnord, Bojnord, Iran.\\
E-mail address:   h.safa@ub.ac.ir \\

Guy R. Biyogmam \\
Department of Mathematics, Georgia College \& State University, Milledgeville, GA, USA.\\
E-mail address:  guy.biyogmam@gcsu.edu \\ \ \\


\begin{thebibliography}{0}
\bibitem{a-p-s} H. Arabyani, F. Panbehkar and H. Safa,  {On the structure of factor Lie algebras}, {\it Bull. Korean Math. Soc.} {\bf 54}(2) (2017), 455--461.

\bibitem{b-m-s} P. Batten, K. Moneyhun and E. Stitzinger, {On characterizing nilpotent Lie algebras by their multipliers,}  {\it Commun. Algebra}
{\bf 24}(14) (1996), 4319--4330.
\bibitem{brk} Y. G. Berkovich, {On the order of the commutator subgroups and the Schur multiplier of a finite $p$-group,} {\it J. Algebra}  {\bf 144}(2) (1991), 269--272.
 \bibitem{Bio} J. C. Bioch, {On $n$-isoclinic groups}, {\it Indag. Math.}  {\bf 38} (1976), 400--407.

\bibitem{b-c1} {G. R. Biyogmam and  J. M. Casas}, {A study of $n$-$\Lie$-isoclinic Leibniz algebras}, {\it J. Algebra Appl.} {\bf 19}(1) (2020) 2050013 (16 pages).

\bibitem{b-c2} {G. R. Biyogmam and  J. M. Casas}, {On ${\Lie}$-isoclinic Leibniz algebras}, {\it J. Algebra}  {\bf 499} (2018), 337--357.
\bibitem{b-c3} {G. R. Biyogmam and  J. M. Casas}, {The $c$-nilpotent Schur ${\Lie}$-multiplier of Leibniz algebras},  {\it J. Geom. Phys.}  \textbf{138} (2019),  55--69.
\bibitem{b-s} {G. R. Biyogmam and  H. Safa}, {On the Schur multiplier and covers of a pair of Leibniz algebras},  {\it J. Lie Theory} (in press).
\bibitem{c-i} {J. M. Casas and M. A. Insua}, {The Schur \Lie-multiplier of Leibniz algebras}, {\it 	Quaest. Math.}  {\bf 41}(7) (2018), 917--936.
\bibitem{c-l-p} J. M. Casas, J. -L. Loday and T. Pirashvili, Leibniz $n$-algebras, {\it Forum Math.} {\bf 14} (2002), 189--207.
\bibitem{c-v} J. M. Casas and T. Van der Linden, {Universal central extensions in semi-abelian categories,} {\it Appl. Categ. Struct.} {\bf 22} (1) (2014), 253--268.
\bibitem{d-s} H. Darabi and F. Saeedi, On the Schur multiplier of $n$-Lie algebras, {\it J. Lie Theory}  {\bf 27}  (2017), 271--281.
\bibitem{lls} G. Ellis,  {On the Schur multipliers of $p$-groups,} {\it Commun. Algebra}  {\bf 27}(9) (1999), 4173--4177.

\bibitem{e-m-s} M. Eshrati, M. R. R. Moghaddam and F. Saeedi, {Some properties on isoclinism in $n$-Lie algebras,} {\it Commun. Algebra}  {\bf 44} (2016), 3005--3019.
\bibitem{Fil} V. T. Filippov,  {$n$-Lie algebras}, {\it Sib. Mat. Zh.} {\bf 26}(6) (1987), 126--140.
\bibitem{gsc} W. Gasch\"utz, J. Neub\"user and T. Yen,{ \"Uber den Multiplikator von $p$-Gruppen,} {\it Math. Z.}  {\bf 100}(2) (1967), 93--96.
\bibitem{grn} J. A. Green, {On the number of automorphisms of a finite group,} {\it Proc. Roy. Soc. London} Ser. A. {\bf 237} (1965), 574--581.
\bibitem{p-h}  {P. Hall}, {The classification of prime-power groups},  {\it J. Reine Angew. Math.} {\bf 182} (1940), 130--141.
\bibitem{h} P. Hardy, On characterizing nilpotent Lie algebras by their multipliers III, {\it Commun. Algebra} {\bf 33} (2005), 4205--4210.
\bibitem{h-s} P. Hardy and E. Stitzinger, On characterizing nilpotent Lie algebras by their multipliers $t(L)=3,4,5,6$, {\it Commun. Algebra}
 {\bf 26} (1998),  3527--3539.
\bibitem{He} {N. S. Hekster}, {On the structure of $n$-isoclinism classes of groups}, {\it J. Pure Appl. Algebra} {\bf 40} (1986), 63--85.

\bibitem{i} I. M. Isaacs, {Derived subgroups and centers of capable groups,} {\it Proc. Amer. Math. Soc.} {\bf 129}(10) (2001), 2853--2859.

\bibitem{l2}  J. -L. Loday, Cohomologie et groupe de Steinberg relatifs, {\it J. Algebra}   {\bf 54} (1978), 178--202.
\bibitem{l1} J. -L. Loday,  {\it Cyclic Homology},   Grundl. Math. Wiss. Bd. {301}, Springer, Berlin, 1992.

\bibitem{l3}  J. -L. Loday, Une version non commutative des alg\`ebres de Lie: les alg\`ebres de Leibniz, {\it Enseign. Math.}   {\bf 39} (1993), 269--293.



\bibitem{m} K. Moneyhun,  Isoclinisms in Lie algebras, {\it Algebras Groups Geom.}  {\bf 11} (1994), 9--22.

\bibitem{n1} S. Nayak, Multipliers of nilpotent Lie superalgebras, {\it Commun. Algebra} {\bf 47}(2)  (2019), 689--705.

\bibitem{n} P. Niroomand,  On dimension of the Schur multiplier of nilpotent Lie algebras, {\it Cent. Eur. J. Math.}  {\bf 9}(1) (2011), 57--64.
\bibitem{PMKh} F. Parvaneh, M. R. R. Moghaddam and A. Khaksar, {Some properties of $n$-isoclinism in Lie algebras}, {\it Ital. J. Pure Appl. Math.} {\bf 28}  (2011), 167--178.

\bibitem{p-z} K. Podoski and B. Szegedy,  {Bounds for the index of the centre in capable groups,} {\it Proc. Amer. Math. Soc.}
{\bf133}(12) (2005), 3441--3445.

\bibitem{r-c} Z. Riyahi and J. M. Casas,  {${\Lie}$-isoclinism of pairs of Leibniz algebras,} {\it Bull. Malays. Math. Sci. Soc.}
{\bf 43}(1) (2020), 283--296.


\bibitem{safa} H. Safa, The Schur multiplier of an $n$-Lie superalgebra,  arXiv:2004.05753v1.

\bibitem{s-b} H. Safa and G. R. Biyogmam,  \Lie-isoclinism in Leibniz $n$-algebras,  (submitted).

\bibitem{s-b2} H. Safa and G. R. Biyogmam,  On \Lie-isoclinic extensions of Leibniz $n$-algebras, (submitted).
\bibitem{sal2} A. R. Salemkar, H. Bigdely and V. Alamian, {Some properties on isoclinism of Lie algebras and covers}, {\it J. Algebra Appl.} {\bf 7}(4) (2008),  507--516.
\bibitem{sal3} A. R. Salemkar, B. Edalatzadeh and H. Mohammadzadeh, {On covers of perfect Lie algebras}, {\it  Algebra Colloq.}  {\bf 18}(3) (2011),  419--427.
\bibitem{s} I. Schur, \newblock \"{U}ber die Darstellung der endlichen Gruppen durch gebrochen lineare Substitutionen, \newblock {\em J. Reine Angew. Math.} \textbf{127} (1904), 20--50.
\bibitem{w} J. Wiegold, {Multiplicators and groups with finite central factor groups}, {\it  Math. Z.} {\bf 89} (1965), 345--347.
\bibitem{zh} X. Zhou, {On the order of Schur multipliers of finite $p$-groups,} {\it Commun. Algebra} {\bf 22}(1) (1994), 1--8.


\end{thebibliography}
\end{document}